\documentclass[12pt]{article}

\usepackage{amsmath,amssymb,amsthm}
\usepackage[colorlinks=true,citecolor=black,linkcolor=black,urlcolor=blue]{hyperref}
\usepackage{graphicx}

\usepackage{cleveref}
\crefname{equation}{}{}
\Crefname{equation}{Equation}{Equations}
\crefname{theorem}{Theorem}{Theorems}
\Crefname{theorem}{Theorem}{Theorems}
\crefname{lemma}{Lemma}{Lemmas}
\Crefname{lemma}{Lemma}{Lemmas}
\crefname{proposition}{Proposition}{Propositions}
\Crefname{proposition}{Proposition}{Propositions}
\crefname{corollary}{Corollary}{Corollaries}
\Crefname{corollary}{Corollary}{Corollaries}
\crefname{conjecture}{Conjecture}{Conjectures}
\Crefname{conjecture}{Conjecture}{Conjectures}
\crefname{section}{Section}{Sections}
\Crefname{section}{Section}{Sections}
\crefname{example}{Example}{Examples}
\Crefname{example}{Example}{Examples}
\crefname{problem}{Problem}{Problems}
\Crefname{problem}{Problem}{Problems}
\crefname{remark}{Remark}{Remarks}
\Crefname{remark}{Remark}{Remarks}
\crefname{figure}{Figure}{Figures}
\Crefname{figure}{Figure}{Figures}
\crefname{question}{Question}{Questions}
\Crefname{question}{Question}{Questions}

\newcommand{\Z}{\mathbb{Z}}

\newtheorem{theorem}{Theorem}

\newtheorem{proposition}[theorem]{Proposition}
\newtheorem{corollary}[theorem]{Corollary}

\newcommand{\doi}[1]{\href{http://dx.doi.org/#1}{\texttt{doi:#1}}}
\newcommand{\arxiv}[1]{\href{http://arxiv.org/abs/#1}{\texttt{arXiv:#1}}}

\title{Approximate {S}teiner $(r-1,r,n)$-systems 
without $3$ blocks on $r+2$ points}

\author{Alexander Sidorenko\\ 
\small\tt sidorenko.ny@gmail.com}

\begin{document}

\maketitle

\begin{abstract}
For a family ${\mathcal F}$ of $r$-graphs, 
let $\mathrm{ex}(n,{\mathcal F})$ denote 
the maximum number of edges 
in an ${\mathcal F}$-free $r$-graph on $n$ vertices. 
Let ${\mathcal F}_r(v,e)$ denote the family 
of all $r$-graphs with $e$ edges and at most $v$ vertices. 
We prove that 
$\mathrm{ex}(n,{\mathcal F}_r(r+1,2) \cup {\mathcal F}_r(r+2,3)) =
(\frac{1}{r} - o(1)) \binom{n}{r-1}$.
\end{abstract}

Let ${\mathcal F}$ be a family of $r$-graphs, 
(that is, $r$-uniform hypergraphs). 
An $r$-graph is called ${\mathcal F}$-{\it free} 
if it does not have a subgraph isomorphic to a member of 
${\mathcal F}$. 
Let $\mathrm{ex}(n,{\mathcal F})$ denote 
the maximum number of edges 
in an ${\mathcal F}$-free $r$-graph on $n$ vertices. 
Let ${\mathcal F}_r(v,e)$ denote the family 
of all $r$-graphs with $e$ edges and at most $v$ vertices. 
In 1973, Brown, Erd\H{o}s and S\'{o}s 
initiated the study of 
$f_r(n,v,e) := \mathrm{ex}(n,{\mathcal F}_r(v,e))$. 
The case $r=2$ had been studied before. 
In particular, Erd\H{o}s \cite{Erdos:1964} proved 
$f_2(n,e+1,e) = \left\lfloor \frac{e-1}{e} n \right\rfloor$.
Cases when the ratio $(er-v)/(e-1)$ is not integer 
were studied in 
\cite{Alon:2006,Erdos:1986,Ge:2017,Ruzsa:1978,Sarkozy:2004,Sarkozy:2005,Shangguan:2019a}.
In the case when $k = (er-v)/(e-1)$ is an integer,
Brown, Erd\H{o}s and S\'{o}s \cite{Brown:1973} 
were able to prove 
\[
  (c_{r,v,e}+o(1)) \: n^k
    \; \leq \;
  f_r(n,v,e)
    \; \leq \;
  (C_{r,v,e}+o(1)) \: n^k
    \;\;\;\;{\rm as}\;\; n \to \infty \: .
\]
The case $e=2$ is equivalent to the packing problem 
and was asymptotically solved by R\"{o}dl \cite{Rodl:1985} 
who proved 
\begin{equation}\label{eq:packing}
  f_r(n,v,2) \; = \; 
  \left(\binom{r}{2r-v}^{-1} \! + o(1)\right) \binom{n}{2r-v}
    \;\;\;\;{\rm as}\;\; n \to \infty \: .
\end{equation}
Recently, the first case with $e=3$, $r \geq 3$ 
was solved by Glock \cite{Glock:2019}: 
$f_3(n,5,3) = (\frac{1}{5}+o(1)) n^2$. 
His result was generalized 
by Shangguan and Tamo \cite{Shangguan:2019} to 
\[
  f_r(n,3r-4,3) \; = \; 
  \left(\frac{1}{r^2-r-1} + o(1)\right) \: n^2 
    \;\;\;\;{\rm as}\;\; n \to \infty \: .
\]
Shangguan and Tamo proved the upper bound
\begin{equation*}
  f_r(n,3r-2k,3) \; \leq \; 
  \left(\frac{1}{\binom{r}{k} - \frac{1}{2}} + o(1)\right)
  \binom{n}{k}
    \;\;\;\;{\rm as}\;\; n \to \infty \: ,
\end{equation*}
which in the case $e=3$, $\: v=r+2$ becomes
\[
  f_r(n,r+2,3) \; \leq \;
  \left(\frac{1}{r - \frac{1}{2}} + o(1)\right) \binom{n}{r-1} 
    \;\;\;\;{\rm as}\;\; n \to \infty \: .
\]
A lower bound, 
obtained from the probabilistic construction of 
Brown, Erd\H{o}s, and S\'{o}s \cite{Brown:1973},
was explicitly calculated in \cite{Shangguan:2019}: 
\[
  f_r(n,r+2,3) \; \geq \;
  \left(\frac{1}{r} 
    \binom{\frac{1}{2}(r+2)(r+1)}{3}^{-1/2} \! + o(1)\right) 
  \binom{n}{r-1} 
    \;\;\;\;{\rm as}\;\; n \to \infty \: .
\]
We will prove a much better lower bound 
which almost matches the upper one: 
\[
  f_r(n,r+2,3) \; \geq \;
  \left(\frac{1}{r} - o(1)\right) \binom{n}{r-1} 
    \;\;\;\;{\rm as}\;\; n \to \infty \: .
\]

Answering a question of Erd\H{o}s \cite{Erdos:1976}, 
two sets of authors in \cite{Bohman:2019} and \cite{Glock:2018} 
independently 
proved that 
for any $m \geq 2$, 
\[
  \mathrm{ex}\Big(n,\: {\mathcal F}_3(4,2) \cup 
                        {\mathcal F}_3(5,3) \cup \ldots \cup 
                        {\mathcal F}_3(m+2,m) \Big) 
  \; = \; \left(\frac{1}{3} + o(1)\right) \binom{n}{2} \; .
\]
It was suggested in \cite[Conjecture~7.2]{Glock:2018} 
that a similar statement must hold 
for all $r \geq 3$, $m \geq 2$, 
namely 
\[
  \mathrm{ex}\Big(n,   {\mathcal F}_r(r+1,2) \cup 
                        {\mathcal F}_r(r+2,3) \cup \ldots \cup 
                        {\mathcal F}_r(r+m-1,m) \Big) 
  = \left(\frac{1}{r} + o(1)\right) \binom{n}{r-1} .
\]
We will prove this statement for $m=3$ 
by constructing
${\mathcal F}_r(r+2,3)$-free 
(or, in terms of \cite{Glock:2018}, $3$-sparse) 
partial Steiner $(r-1,r,n)$-systems 
with $(\frac{1}{r} - o(1)) \binom{n}{r-1}$ blocks (edges). 
The case $r=3$ is trivial, 
because any ${\mathcal F}_3(4,2)$-free system 
is also ${\mathcal F}_3(5,3)$-free. 
We will provide different constructions for 
$r=4$ (\cref{th:6_3}) 
and $r \geq 5$ (\cref{th:r+2_3}).

We define the {\it degree} of a subset of vertices 
in an $r$-graph 
as the number of edges that contain this subset.

\begin{theorem}\label{th:6_3}
For any even $n$, there exists an 
$({\mathcal F}_4(5,2) \cup {\mathcal F}_4(6,3))$-free 
$4$-graph with $n$ vertices and 
$\frac{1}{4} \binom{n}{3} - \frac{n(n-2)}{8}$ edges.
\end{theorem}

\begin{proof}[\bf{Proof}]
Consider a $4$-graph $H(n)$ 
whose vertices are elements of $\Z_n$, 
and $4$ distinct vertices $i,j,k,l$ form an edge when 
$i+j+k+l \equiv 1 \pmod{n}$. 
Then the degree of a triple $\{i,j,k\}$ is either $1$ 
(when $i+j+2k,\: i+2j+k,\: 2i+j+k \not\equiv 1 \pmod{n})$, 
or $0$. 
For each $i\in\Z_n$, 
there are exactly $\frac{1}{2}(n-2)$ pairs $\{j,k\}$ such that 
$j \neq i$, $k \neq i$, $j \neq k$, 
and $2i+j+k \equiv 1 \pmod{n}$. 
Hence, the number of triples of degree $0$ is
$\frac{n(n-2)}{2}$, 
and consequently, the number of edges in $H(n)$ is 
$\frac{1}{4}\Huge(\binom{n}{3} - \frac{n(n-2)}{2}\Huge)$. 

The only $4$-graph in 
${\mathcal F}_4(5,2) \cup {\mathcal F}_4(6,3)$ 
without triples of degree $\geq 2$ 
is formed by $3$ edges $A,B,C$ such that 
$A = D \cup E$, $B = D \cup F$, $C = E \cup F$, 
where $D,E,F$ are disjoint pairs of vertices. 
If they were present in $H(n)$, we would get 
\[
  1+1+1 \; \equiv \; 
  \sum_{i \in A} i + \sum_{i \in B} i + \sum_{i \in C} i
  \; \equiv \: 2 
  \!\!\!\! \sum_{i \in D \cup E \cup F}\! i  \pmod{n} \: ,
\]
which is impossible (as $n$ is even).
\end{proof}

\begin{proposition}\label{th:2_d}
When $r>4$ and $n=2^d$, 
there exists an 
$({\mathcal F}_r(r+1,2) \cup {\mathcal F}_r(r+2,3))$-free 
$r$-graph with $n$ vertices and 
$\frac{1}{r} \binom{n}{r-1} - O(n^{r-2})$ edges.
\end{proposition}

\begin{proof}[\bf{Proof}]
Consider an $r$-graph $H_d$ 
whose vertices are elements of $\Z_2^d$, 
where a set $A$ of $r$ elements is an edge 
if the sum of them is zero,
and $A$ does not contain zero-sum subsets of sizes $4$ and $r-4$.
It is easy to see that 
the degree of any $(r-1)$-subset of vertices is at most $1$. 
An $(r-1)$-subset has degree $0$ only when 
it contains a zero-sum subset of size $r-2$, $r-4$ or $4$. 
For a fixed $k$, the number of zero-sum subsets of size $k$ 
is $\frac{1+o(1)}{n}\binom{n}{k}$, 
so the number of $(r-1)$-subsets that contain 
a zero-sum subset of size $k$ is at most 
$\frac{1+o(1)}{n}\binom{n}{k}\binom{n-k}{r-1-k} = O(n^{r-2})$. 
Hence, the number of $(r-1)$-subsets of degree $0$ is 
$O(n^{r-2})$, and the number of edges in $H_d$ is 
$\frac{1}{r}\left(\binom{n}{r-1} - O(n^{r-2})\right)$. 
Obviously, $H_d$ is ${\mathcal F}_r(r+1,2)$-free. 
We need to show that $H_d$ is ${\mathcal F}_r(r+2,3)$-free.
Suppose, there exist an $(r+2)$-subset $X$ 
and $3$ distinct edges $A,B,C \subset X$. 
Since 
$|A \cup B|=|A|+|B|-|A \cap B| \geq 2r - (r-2) = r+2 = |X|$, 
we get $A \cup B = X$, and similarly, $A \cup C = B \cup C = X$.
Set 
$\overline{A}=X-A$,\: 
$\overline{B}=X-B$,\: 
$\overline{C}=X-C$. 
Then $|\overline{A}|=|\overline{B}|=|\overline{C}|=|X|-r=2$, 
and $\overline{A}$, $\overline{B}$, $\overline{C}$ 
are pairwise disjoint, because 
$\overline{A} \cap \overline{B} = X - (A \cup B) = \emptyset$. 
Set $D = X- (\overline{A}\cup\overline{B}\cup\overline{C})$. 
Then 
$A = D \cup \overline{B} \cup \overline{C}$,\: 
$B = D \cup \overline{A} \cup \overline{C}$,\: 
$C = D \cup \overline{A} \cup \overline{B}$,\: 
$|D| = (r+2) - 3 \cdot 2 = r-4$,
and we get 
\[
  0 \; = \; \sum_{z \in A} z \: 
  + \: \sum_{z \in B} z \: 
  + \: \sum_{z \in C} z
  \; = \; 2 \!\!\!\! \sum_{z \in \overline{A} \cup \overline{B} \cup \overline{C}} \!\!\!\!\; z 
    \: + \: 3 \sum_{z \in D} z 
      \; = \; \sum_{z \in D} z \; ,
\]
so $D$ is a zero-sum subset of size $r-4$ in edge $A$, 
a contradiction.
\end{proof}

\begin{proposition}\label{th:m2_d}
When $r>4$ and $n = m 2^d$, there exists an 
$({\mathcal F}_r(r+1,2) \cup {\mathcal F}_r(r+2,3))$-free 
$r$-graph $H_{m,d}$ with $n$ vertices and 
$(\frac{1}{r} - O(\frac{1}{2^{d}})) \binom{n}{r-1}$ edges.
\end{proposition}

\begin{proof}[\bf{Proof}]
The case $m=1$ is covered by \cref{th:2_d}. 
Assume $m\geq 2$ and consider an $r$-graph $H_{m,d}$
whose vertices are elements of $\Z_m\oplus\Z_2^d$, 
and $A=\{(x_1,y_1),(x_2,y_2),\ldots,(x_r,y_r)\}$ 
is an edge in $H_{m,d}$ 
when $x_1+x_2+\ldots+x_r=0$ in $\Z_m$ and
$A'=\{y_1,y_2,\ldots,y_r\}$ is an edge 
in $r$-graph $H_d$ from \cref{th:2_d}. 
(We say that $A'$ is the {\it shadow} of $A$ in $H_d$.) 
The number of edges in $H_{m,d}$ 
is $m^{r-1}$ times that of $H_d$. 
It is easy to see that $H_{m,d}$ is 
${\mathcal F}_r(r+1,2)$-free. 
We need to check that it is also 
${\mathcal F}_r(r+2,3)$-free. 
Suppose, there exist three distinct edges $A,B,C$ 
such that $|A \cup B \cup C| \leq r+2$. 
As $H_d$ is ${\mathcal F}_r(r+2,3)$-free, 
at least two of these three edges 
must have the same shadow in $H_d$. 
Let $A=\{(x_1',y_1),(x_2',y_2),\ldots,(x_r',y_r)\}$ 
and $B=\{(x_1'',y_1),(x_2'',y_2),\ldots,(x_r'',y_r)\}$. 
Since $x_1'+x_2'+\ldots+x_r'=0$ 
and $x_1''+x_2''+\ldots+x_r''=0$, 
vectors $(x_1',x_2',\ldots,x_r')$ 
and $(x_1'',x_2'',\ldots,x_r'')$ 
must differ in at least two places. 
Since $|A \cup B| \leq r+2$, 
the same vectors 
must differ in at most two places. 
Therefore, they differ in exactly two places, 
and $|A \cup B| = r+2$. 
Then $C$ must have the same shadow as $A$ and $B$: 
$C=\{(x_1''',y_1),(x_2''',y_2),\ldots,(x_r''',y_r)\}$. 
Similarly, $(x_1''',x_2''',\ldots,x_r''')$ 
must differ in exactly two places from 
$(x_1',x_2',\ldots,x_r')$ 
and from $(x_1'',x_2'',\ldots,x_r'')$. 
As $|A \cup B \cup C| = r+2$, 
it must be the same places where $A$ and $B$ differ,
but this is impossible.
\end{proof}

\begin{theorem}\label{th:r+2_3}
When $r>4$, there exists an 
$({\mathcal F}_r(r+1,2) \cup {\mathcal F}_r(r+2,3))$-free 
$r$-graph with $\leq n$ vertices and 
$(\frac{1}{r} - O(\frac{1}{\sqrt{n}})) \binom{n}{r-1}$ edges 
as $n\to\infty$.
\end{theorem}

\begin{proof}[\bf{Proof}]
For a given $n$, select $d=\lfloor\log_2\sqrt{n}\rfloor$ 
and $m=\lfloor n/2^d\rfloor$. 
Then $0 \leq n - m 2^d < 2^d \leq \sqrt{n}$. 
By \cref{th:m2_d}, there exists an 
$({\mathcal F}_r(r+1,2) \cup {\mathcal F}_r(r+2,3))$-free 
$r$-graph with $m 2^d$ vertices and 
$(\frac{1}{r} - O(\frac{1}{2^{d}})) \binom{m2^d}{r-1}$ edges. 
Notice that $\frac{1}{2^{d}} = O(\frac{1}{\sqrt{n}})$ and 
$\binom{m2^d}{r-1} = (1 - O(\frac{1}{\sqrt{n}})) \binom{n}{r-1}$ 
as $n\to\infty$.
\end{proof}

Any ${\mathcal F}_r(r+1,2)$-free $r$-graph 
has at most $\frac{1}{r} \binom{n}{r-1}$ edges 
(because every $(r-1)$-subset of vertices 
has degree at most $1$). 
Hence, \cref{th:6_3,th:r+2_3} imply

\begin{corollary}\label{th:main}
When $r \geq 3$ and $n\to\infty$, 
\[
  \mathrm{ex}\Big(n,\: {\mathcal F}_r(r+1,2) \cup 
                     {\mathcal F}_r(r+2,3)\Big)
    \; = \;
  \left(\frac{1}{r}-o(1)\right)\binom{n}{r-1}
    \; .
\]
\end{corollary}

Simplicity of our constructions of approximate Steiner 
$(r-1,r,n)$-systems 
(comparing with R\"{o}dl's probabilistic proof 
of \cref{eq:packing}) 
can be explained by the fact that their main parameters, 
$r-1$ and $r$, differ just by $1$.

\vspace{4mm}
{\bf Acknowledgments.}
The author would like to thank two anonymous referees 
for their valuable comments and suggestions.

\end{document}